\documentclass{nyj}

\usepackage{amsmath}
\usepackage{amscd}
\usepackage{amssymb}
\numberwithin{equation}{section}

\input{prepictex}
\input{pictex}
\input{postpictex}

\newtheorem{theorem}{Theorem}[section]

\newtheorem{lemma}[theorem]{Lemma}

\newtheorem{corollary}[theorem]{Corollary}

\newtheorem{proposition}[theorem]{Proposition}

\theoremstyle{definition}

\newtheorem{definition}[theorem]{Definition}

\newtheorem{example}[theorem]{Example}

\newtheorem{remark}[theorem]{Remark}

\newsymbol\kk 207C

\renewcommand{\AA}{{\mathbb A}}

\newcommand{\FF}{{\mathbb F}}

\newcommand{\HH}{{\mathbb H}}

\newcommand{\NN}{{\mathbb N}}

\newcommand{\QQ}{{\mathbb Q}}

\newcommand{\ZZ}{{\mathbb Z}}

\newcommand{\ff}{{\mathfrak f}}

\newcommand{\cA}{{\mathcal A}}

\newcommand{\cH}{{\mathcal H}}

\newcommand{\cR}{{\mathcal R}}

\newcommand{\cV}{{\mathcal V}}

\newcommand{\cY}{{\mathcal Y}}

\newcommand{\id}{{\bf 1}}

\newcounter{picture}
\renewcommand{\thefigure}{\thepicture}

\newcommand{\D}{{\Delta}}
\newcommand{\G}{{\Gamma}}

\newcommand{\Om}{{\Omega}}

\newcommand{\g}{{\gamma}}
\newcommand{\s}{{\sigma}}
\newcommand{\w}{{\omega}}

\newcommand{\aut}{{\text{\rm Aut}}}
\newcommand{\traut}{{\text{\rm Aut}_{\text{\rm tr}}(\D)}}
\newcommand{\autD}{{\text{\rm Aut}(\D)}}

\newcommand{\sh}{{\sigma}}
\newcommand{\tA}{{\widetilde{A}}}
\newcommand{\st}{{\; ;\;}}
\newcommand{\axa}{{\tA_1\times\tA_1}}
\newcommand{\zomatrix}{{$\{0,1\}$-matrix }}
\newcommand{\zomatrices}{{$\{0,1\}$-matrices }}

\DeclareMathOperator{\rank}{rank}

\DeclareMathOperator{\coker}{coker}

\DeclareMathOperator{\tors}{tors}
\DeclareMathOperator{\conv}{conv}
\begin{document}

\title[Groups acting on products of trees]
{Groups acting on products of trees, tiling systems and analytic K-theory}

\date{August 4, 2002}
\author{Jason S. Kimberley}
\author{Guyan Robertson}
\address{Mathematics Department, University of Newcastle, Callaghan, NSW
2308, Australia}
\email{guyan@maths.newcastle.edu.au}
\email{Jason.S.Kimberley@telstra.com}
\thanks{This research was funded by the Australian Research Council.
The second author is also grateful for the support of the University
of Geneva.}
\begin{abstract}
Let  $T_1$ and $T_2$ be homogeneous trees of even degree $\ge 4$. A BM group $\Gamma$ is a torsion free discrete subgroup of $\aut (T_1) \times \aut (T_2)$
which acts freely and transitively on the vertex set of $T_1 \times T_2$.
This article studies dynamical systems associated with BM groups.
A higher rank Cuntz-Krieger algebra $\mathcal A(\G)$ is associated both with a 2-dimensional tiling system and with a boundary action of a BM group $\Gamma$. An explicit expression is given for the K-theory of $\mathcal A(\G)$. In particular $K_0=K_1$.
A complete enumeration of possible BM groups $\G$ is given for a product homogeneous trees of degree 4, and the K-groups are computed.
\end{abstract}

\maketitle

\section{Introduction.}

The structure of a group which acts freely and cocompactly on a tree
is well understood. Any such a group is a finitely generated free group.
By way of contrast, a group which acts in a similar manner on a product of
trees can have remarkably subtle properties.  For example, M. Burger and
S. Mozes  \cite{bm, bm2} have proved rigidity and arithmeticity results
analogous to the theorems of Margulis for lattices in semisimple Lie groups.

This article will consider a discrete subgroup $\Gamma$ of  $\aut (T_1) \times \aut (T_2)$
where $T_1$, $T_2$ are homogeneous trees of finite degree.
In addition, we require that $\Gamma$ is torsion free and acts freely and
transitively on the vertex set of $T_1 \times T_2$. For simplicity, refer to such a group as a BM group.  BM groups were used in \cite{bm} to exhibit the first known examples
of finitely presented, torsion free, simple groups.

A product of two trees may be regarded as the 1-skeleton of an affine building whose
2-cells are euclidean squares. A BM group $\Gamma$ acting freely and
transitively on the vertex set of $T_1 \times T_2$ defines a
2-dimensional tiling system. Associated to this tiling system there is a $C^*$-algebra $\cA$, which is called a rank-2 Cuntz-Krieger algebra in \cite{rs2}. This algebra is
isomorphic to a crossed product $C^*$-algebra $\cA(\G)$ arising from a boundary action of $\Gamma$.
It follows from the results of \cite{rs2} that $\cA(\G)$ is purely infinite, simple, unital  and nuclear,
and is therefore itself classified by its K-theory. This provides the motivation
for us to examine the K-theory of these examples in some detail. In Theorem \ref{mainK} we obtain an explicit expression for the K-theory of $\cA(\G)$ analogous to that of \cite{rs3} for algebras associated with $\tA_2$ buildings. In particular $K_0=K_1$ for this algebra. In Proposition \ref{orderof1}, the class of the identity in $K_0$ is shown to be a torsion element.

In Section \ref{exs}, these issues are examined for several explicit groups.
In Section \ref{bm2by2} a complete list is given of all BM groups acting on $T_1 \times T_2$, where $T_1$, $T_2$ are homogeneous trees of degree four. The abelianizations and K-groups are also computed.

After this article was submitted, we became aware of the work of Diego Rattaggi \cite{rat}, which undertakes a detailed analysis of BM groups, including extensive computations with explicit presentations.  We are grateful to him for several helpful comments on this article.

\section{Products of trees and their automorphisms.}%

Given a homogeneous tree~$T$, there is a type map~$\tau$ defined on the
vertices of~$T$ and taking values in $\ZZ/2\ZZ$.
To see this, fix a vertex $v_0\in T$ and define
\[
\tau(v)=  d(v_0,v) \pmod{2},
\]
where $d(u,v)$ denotes the usual graph distance between vertices of the tree.
The type map partitions the set of vertices into two classes so that two vertices are in the same class if and only if the distance between them is even. Thus the type map is independent of $v_0$, up to addition of $1 \pmod{2}$. Since any automorphism of the tree preserves distances between vertices this observation proves
\begin{lemma}\label{lem1}
For each automorphism $g$ of $T$ there exists $i\in\ZZ/2\ZZ$ such that, for every vertex $v$, $\tau(gv)= \tau(v)+i $ 
\end{lemma}

Suppose that~$\D$ is the $2$ dimensional cell complex associated with  a product $T_1\times T_2$ of homogeneous trees. Then $\D$ is an affine building of type $\axa$ in a natural way \cite{ron}.
Write $u=(u_1,u_2)$ for a generic vertex of $\D$.
There is a type map~$\tau$ on the vertices of~$\D$ where
$$
\tau(v)=
(\tau(v_1),\tau(v_2))\in\ZZ/2\ZZ\times\ZZ/2\ZZ.
$$
We say that an automorphism $g$ of~$\D$ is {\bf
type-rotating} if there exists
$(i_1,i_2)\in\ZZ/2\ZZ\times\ZZ/2\ZZ$
such that, for each vertex $v$,
$$
\tau(gv)=(\tau(v_1)+i_1,\tau(v_2)+i_2).
$$
The chambers of $\D$ are geometric squares and each chamber has exactly one vertex of each type.
We denote by $\traut$ the group of type rotating automorphisms of $\D$.

\begin{lemma} An automorphism $g$ of~$\D$ is
type-rotating if and only if it is a Cartesian
product of automorphisms of the two trees.
\end{lemma}

\begin{proof}
 If we have
\[
g(u_1,u_2)=(g_1 u_1,g_2 u_2)
\]
for some type rotating automorphisms $g_i$ of $\ T_i$ then it follows from Lemma \ref{lem1} that $g$ is type rotating.
Conversely suppose that $g$ is type rotating.
Let $(u_1,u_2)$ and $(u_1,u_2')$ be neighbouring vertices in~$\D$.
Then~$g(u_1,u_2)=(x_1,x_2)$ and $g(u_1,u_2')=(x_1',x_2')$ are neighbouring vertices
in~$\D$ and the type-rotating assumption on~$g$ means that
$\tau(x_1)=\tau(x_1')$.  Since neighbouring vertices in~$\ T_1$ have
distinct types we must have $x_1=x_1'$. By induction on~$d(u_2,u_2')$, we see
that the first coordinate of~$g(u_1,u_2)$ is independent
of~$u_2\in\ T_2$. Similarly, the second coordinate of~$g(u_1,u_2)$ is
independent of~$u_1\in\ T_1$. Thus there exist maps~$g_1$ of~$\ T_1$ and
$g_2$ of~$\ T_2$ such that~$g(u_1,u_2)=(g_1 u_1,g_2 u_2)$. Since~$g$ is an automorphism of $\D$
it follows that each $g_i$ is an automorphism of~$\ T_i$. Thus $g = g_1\times g_2$ for some
automorphisms~$g_i$ of~$\ T_i$.
\end{proof}

\begin{corollary} $\traut = \aut (T_1) \times \aut (T_2)$.
\end{corollary}

An {\bf apartment} in~$\D$ is a subcomplex isomorphic to the plane tessellated by squares. See \cite[p. 184]{ron} for some comments on this and alternative ways of looking at apartments.
Denote by~$ \cV$ the vertex set of~$\D$. Any two vertices~$u,v\in \cV$ belong
to a common apartment. The convex hull, in the sense of buildings, between
two vertices $u$ and $v$ is the subset of an apartment containing $u$ and $v$ depicted in Figure~\ref{convex hull}.
\refstepcounter{picture}
\begin{figure}[htbp]
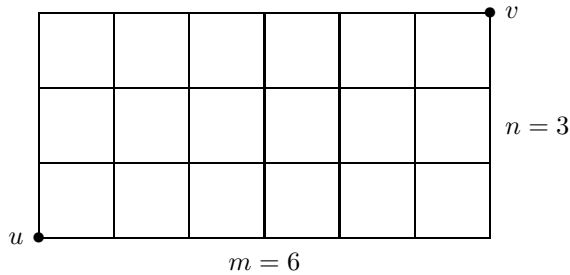
\label{convex hull}
\hfil
\centerline{
\beginpicture
\setcoordinatesystem units <1cm, 1cm>
\setplotarea  x from -6 to 6,  y from -1 to 2.2
\putrule from -3 -1 to 3 -1
\putrule from -3 0 to 3 0
\putrule from -3 1 to 3 1
\putrule from -3 2 to 3 2
\putrule from -3 -1 to  -3 2
\putrule from -2 -1 to  -2 2
\putrule from -1 -1 to  -1 2
\putrule from 0 -1 to  0 2
\putrule from 1 -1 to  1 2
\putrule from 2 -1 to  2 2
\putrule from 3 -1 to  3 2
\put {$u$} [r]     at   -3.2 -1
\put {$\bullet$} at -3 -1
\put {$v$} [l]     at   3.2 2
\put {$\bullet$} at 3 2
\put {$m=6$} [t]     at   0 -1.2
\put {$n=3$} [l]     at   3.2  0.5
\endpicture
}
\hfil
\caption{Convex hull of two vertices.}
\end{figure}
The convex hull of $u$ and $v$ is contained in every apartment of $\D$ which contains $u$ and $v$.

Define the {\bf distance}, $d(u,v)$, between $u$ and $v$ to be the
graph theoretic distance on the one-skeleton of~$\D$. Any path from
$u$ to $v$ of length $d(u,v)$ lies in their convex hull, and the union
of the vertices in such paths is exactly the set of vertices in the
convex hull.

We define the {\bf shape}~$\sh(u,v)$ of the ordered pair of vertices
$(u,v)\in \cV\times \cV$ to be the pair $(m,n)\in\NN\times\NN$ as
indicated in Figure~\ref{convex hull}.  Note that $d(u,v)=m+n$. The
components of $\sh(u,v)$ indicate the relative contributions to
$d(u,v)$ from the two factors.  If $v=(v_1,v_2)$ and $w=(w_1,w_2)$ are vertices of $\D$,
the {\bf shape} from $v$ to $w$ is
\[
\sh(v,w)=\left( d(v_1,w_1), d(v_2,w_2)\right)
\]
where $d$ denotes the usual graph-theoretic distance on a tree.
An edge in $\D$ connects the vertices $v$ and $w$ if $\sh(v,w)=(0,1)$ or
$\sh(v,w)=(1,0)$.

\begin{lemma}\label{shape decomposition for vertices}
Suppose $m_1,m_2,n_1,n_2\in\NN$ and $\sh(u,w)=(m_1+m_2,n_1+n_2)$ for
vertices $u,w\in \cV$. Then there is a unique vertex $v\in \cV$ such that
\[
\sh(u,v)=(m_1,n_1)\quad\text{ and }\quad \sh(v,w)=(m_2,n_2).
\]
\end{lemma}
\begin{proof}
Such a ~$v\in \cV$ satisfies $d(u,w)=d(u,v)+d(v,w)$ so it must lie in the
convex hull of~$u$ and~$w$. Inside the convex hull existence and uniqueness
of~$v$ are clear.
\end{proof}

It is a direct consequence of the definitions that every type-rotating
automorphism~$g\in\autD$ preserves shape in the sense that
$\sh(gu,gv)=\sh(u,v)$ for all~$u,v\in \cV$.

\bigskip

\section{Groups which act freely and transitively on the vertices of $\D$}\label{simply transitive actions}

Suppose that~$\G\leq\traut$ acts freely and transitively on the vertex set $\cV$.
Fix any vertex~$v_0\in \cV$ and let
\[
N=\{ a\in\G\st d(v_0,av_0)=1\}.
\]
The {\bf Cayley graph} of~$\G$ constructed via right multiplication with
respect to the set~$N$ has~$\G$ itself as its vertex set and has
$\{(c,ca)\st c\in\G,a\in N\}$ as its edge set.
There is a natural action of~$\G$ on its Cayley graph via left multiplication.
Using the convention that an undirected edge between vertices~$u$ and~$v$ in
a graph represents the pair of directed edges~$(u,v)$ and~$(v,u)$, it is
immediate that the $\G$-map $c\mapsto cv_0$ from~$\G$ to~$\D$ is an
isomorphism between the Cayley graph of~$\G$ and the one-skeleton of~$\D$.
In this way we identify~$\G$ with the vertex set~$ \cV$ of the $\D$.
Connectivity of the $\D$ implies that~$N$ is a generating set for~$\G$.

It is traditional to label the directed edge~$(c,ca)$ with the
generator~$a\in N$. More generally, to the pair $(c,d)\in\G\times\G$ we
assign the
label~$c^{-1}d$. Equivalently, to the pair~$(c,cd)$ we assign the
label~$d\in\G$. Suppose this label is written as a product of generators;
$d=a_1\cdots a_j$. Then there is a path $(c,ca_1,ca_1a_2,\ldots,cd)$
from~$c$ to~$cd$ whose successive edges are labelled $a_1,\ldots,a_j$.
The left translate of~$(c,cd)$ by $b\in\G$ is~$(bc,bcd)$ and also carries
the label~$d$. Conversely, any pair $(c',c'd)$ which carries the label~$d$
is the left translate by~$b=c'c^{-1}$ of~$(c,cd)$. Thus two pairs carry
the same label if and only if one is the left translate of the other.

We define a shape function on~$\G$ by
\[
\sh(b)=\sh(v_0,bv_0)
\]
for $b\in\G$. The pair~$(c,cb)$ has label~$b$ and its shape, defined via
the identification of the Cayley graph and the one-skeleton of~$\D$, is
\[
\sh(c,cb)=\sh(cv_0,cbv_0)=\sh(v_0,bv_0)=\sh(b).
\]
A different choice of~$v_0$ leads to a shape function on~$\G$ which
differs from the first by an inner automorphism
of~$\G$.

\begin{definition} \cite{rrs}
Suppose that the group $\G\leq\traut$  acts
freely and transitively on $ \cV$.  Then $\G$~is called an
{\bf {\boldmath$\tA_{1}\times\tA_{1}$} group}.
\end{definition}

Consider an $\tA_1\times\tA_1$ group $\G$. Fix a vertex $v_0=(v_1,v_2)\in \cV$ and
suppose that~$g=g_1\times g_2\in\G$. Recall that
\[
\sh(g)=\sh(v_0,gv_0)= \left( d(v_1,g_1v_1), d(v_2,g_2v_2)\right).
\]
Consider the generating set
\[
N=\{ g\in\G\st d(v_0,gv_0)=1\}
\]
of~$\G$. Let
\begin{equation}\label{AB}
A=\{ a\in\G\st \sh(a)=(1,0)\}
\quad \text{ and }\quad
B=\{ b\in\G\st \sh(b)=(0,1)\} .
\end{equation}

\begin{lemma}\label{normal form}
Each element~$g\in\G$ has a unique reduced expression of the form
\[
g=a_1\cdots a_mb_1\cdots b_n
\]
and of the form
\[
g=b^\prime_1\cdots b^\prime_na^\prime_1\cdots a^\prime_m
\]
for some $a_i, a^\prime_i\in A$ and $b_i, b^\prime_i\in B$.
Moreover $\sh(g)=(m,n)$.

\begin{proof}
This follows immediately from Lemma \ref{shape decomposition for vertices}.
\end{proof}
\end{lemma}

\bigskip

In \cite[Section 1]{bm}, M. Burger and S. Mozes constructed a class of groups which act freely and transitively on the vertices of a product of trees. It is convenient to refer to these groups as BM groups.
Our aim now is to show that the class of BM groups
coincides with the class of torsion free $\axa$ groups.

\begin{definition}\label{def BM group}
\cite[Section 1]{bm}
A {\bf BM group} is defined as follows.
Choose sets $A, B$, with $|A|=m, |B|=n$ where $m, n \ge 4$ are even integers. Choose fixed point free involutions
$a \mapsto a^{-1}$, $b \mapsto b^{-1}$ on $A, B$ respectively and a subset $\cR \subset A\times B\times B\times A$
with the following properties.
\begin{itemize}
\item[(i)] If $(a,b,b',a')\in \cR$ then each of
$(a^{-1},b',b,{a'}^{-1}),({a'}^{-1},{b'}^{-1},b^{-1},a^{-1})$,
and $(a',b^{-1},{b'}^{-1},a)$ belong to $\cR$.
\item[(ii)] All four $4$-tuples in (i) are distinct. Equivalently
$(a,b,b^{-1},a^{-1})\notin \cR$ for $a\in A, b\in B$.
\item[(iii)] Each of the four projections of $\cR$ to a subproduct of the form $A\times B$ or $B\times A$ is bijective. Equivalently at least one such projection is bijective.
\end{itemize}
A {\bf BM square} is defined to be a set of four distinct tuples as
in (ii), that is four element subsets of $A\times B\times B\times A$ of the form
$$\{(a,b,b',a'), (a^{-1},b',b,{a'}^{-1}) ,({a'}^{-1},{b'}^{-1},b^{-1},a^{-1}), (a',b^{-1},{b'}^{-1},a)\}.$$
\end{definition}
If $(a,b,b',a')\in \cR$ then write $ab\square b'a'$.
\refstepcounter{picture}
\begin{figure}[htbp]
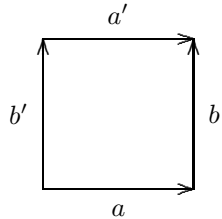
\label{geometric square}
\hfil
\centerline{
\beginpicture
\setcoordinatesystem units <1cm, 1cm>
\setplotarea  x from -6 to 6,  y from -1 to 1.2
\putrule from -1 -1 to 1 -1
\putrule from -1  1 to 1  1
\putrule from -1 -1 to -1 1
\putrule from  1  -1 to 1  1
\put {$b'$} [r]     at   -1.2 0
\put {$b$} [l]     at    1.2 0
\put {$a$} [t]     at   0 -1.2
\put {$a'$} [b]     at   0  1.2
\arrow <6pt> [.3,.67] from    0.8 -1   to    1 -1
\arrow <6pt> [.3,.67] from    0.8  1   to  1  1
\arrow <6pt> [.3,.67] from    -1  0.8   to  -1 1
\arrow <6pt> [.3,.67] from   1  0.8    to   1  1
\endpicture
}
\hfil
\caption{A geometric square.}
\end{figure}

A BM group $\G$ has presentation
\begin{equation}\label{presentation}
\G = \langle A\cup B \st ab=b'a' {\rm \quad whenever \quad} ab\square b'a' \rangle
\end{equation}

In a subsequent article  \cite{bm2}, a set of objects $(A,B,a\mapsto a^{-1},b\mapsto b^{-1}, \cR)$ satisfying the above conditions is called a {\bf VH-datum}.

\bigskip

\begin{theorem}\label{BMA} A group $\G$ is a BM group if and only if it is a torsion free $\axa$ group.
\end{theorem}
\begin{proof}
Suppose that $\G$ is a torsion free $\axa$ group with generating set $N=A\cup B$, where
$A$ and $B$ are described by equations (\ref{AB}). The map $c\mapsto c^{-1}$ on $N$ is fixed point free since $\G$ is torsion free. Define $\cR$ to be the set of $4$-tuples
$(a,b,b',a')\in A\times B\times B\times A$  such that $ab=b'a'$. Condition (i) for a BM group is clearly satisfied. To verify condition (ii) note that if
$(a,b,b^{-1},a^{-1})\in \cR$ then $(ab)^2=1$, contradicting the assumption that $\G$ is torsion free. Condition (iii) follows immediately from Lemma \ref{normal form}.

We now prove the converse. Given a BM group $\G$ we may construct as in \cite{bm} a cell complex $\cY$ whose fundamental group is $\G$. The complex $\cY$ has one vertex $v$ and the cells are geometric squares as in Figure \ref{geometric square}. There are $|A||B|/4$ such cells whose four edges form a bouquet of four loops meeting at $v$. The boundary labels of the directed edges are elements of $A\cup B$ and edges with the same label are identified in the complex. Definition \ref{def BM group}(ii) says that none of the cells of $\cY$ is a projective plane.

By definition, the {\bf link} of the vertex $v$ in $\cY$ is the graph $Lk(v,\cY)$ whose vertices are in 1-1 correspondence with the half-edges incident with $v$ and whose edges are in 1-1 correspondence with the corners incident at $v$. In our setup the link $Lk(v,\cY)$ is a complete bipartite graph with vertex set $A\cup B$ and an edge between each element of $A$ and each element of $B$. Intuitively, completeness of this bipartite graph means that there are no ``missing corners''.  It follows from \cite[Theorem 10.2]{bw} that the universal cover of $\cY$ is a product of homogeneous trees $\D=T_1\times T_2$, where $T_1$ has valency $|A|$ and $T_2$ has valency $|B|$. (In the terminology of \cite{bw}, $\cY$ is said to be a complete $\cV\cH$ complex.) Elements of $\G$ correspond to edge paths in $\cY$. By \cite[Lemma 4.3]{bw} each $g\in \G$ can be expressed uniquely in each of the normal forms
$
g=a_1\cdots a_mb_1\cdots b_n = b^\prime_1\cdots b^\prime_na^\prime_1\cdots a^\prime_m
$,
for some $a_i, a^\prime_i\in A$ and $b_i, b^\prime_i\in B$.
The 1-skeleton of the universal covering space $T_1\times T_2$ may therefore be identified with the Cayley graph of $\G$ with respect to the generating set $A\cup B$.
Thus $\G \subset \traut = \aut T_1 \times \aut T_2$, and $\cY=\G\backslash\D$.

Let $\G$ be a BM group with presentation (\ref{presentation}). In view of the preceding discussion we need only show that $\G$ is torsion free. The argument for this is well known \cite[VI.5, p.161, Theorem]{bk}, and it was shown to us by Donald Cartwright, in the context of $\tA_2$ groups.
Suppose that $1\ne x \in \G$ with $x^n=1$ for some integer $n>0$. Let $C(x)$ denote the cyclic group generated by $x$. Fix a vertex $v_0$ of the $1$ skeleton of $\D=T_1\times T_2$.
Then $\G v_0$ is the set of vertices of $\D$. Now the set $C(x) v_0$ is a bounded $C(x)$-stable subset of $\D$. Since the complete metric space $\D$ satisfies the negative curvature condition of \cite[VI.3b]{bk}, it follows from the Bruhat-Tits Fixed Point Theorem that there is a $C(x)$-fixed point $p\in\D$. Since the action of $\G$ is free on the vertices of $\D$, $p$ cannot be a vertex.
Thus $p$ lies in the interior of an edge $E$ or a square $S$ in $\D$, and either $E$ or $S$ is invariant under $C(x)$.
By considering $g^{-1}xg$ for suitable $g\in \G$, we may suppose that one vertex of  $E$ [respectively $S$] is $v_0$.

\noindent {\bf Case 1.}\ $E$ is invariant under $C(x)$. It follows that $E$ has endpoints $v_0$,$xv_0$ and that $x\in A\cup B$ satisfies $x^2=1$, contradicting the definition of a BM group. We therefore reduce to

\noindent {\bf Case 2.}\ $S$ is invariant under $C(x)$, where $S$ is the square illustrated in Figure \ref{square S}. Then $xv_0\ne v_0$ (since the action is free) and so $xv_0$ is one of the other three vertices of $S$. Thus $x=a$ or $x=b'$ or $x=ab$, where $a\in A$,$b,b'\in B$. If $x=a$ then $xav_0=a^2v_0$ is a vertex of $S$, which is impossible. Similarly $x\ne b'$. Thus $x=ab$. Again $x^2v_0=ababv_0$ is a vertex of $S$.
According to Lemma \ref{normal form}, the only way this can happen is if $abab=1$. However this contradicts condition (ii) in the Definition \ref{def BM group}.
This completes the proof of Theorem \ref{BMA}.
\end{proof}

\refstepcounter{picture}
\begin{figure}[htbp]
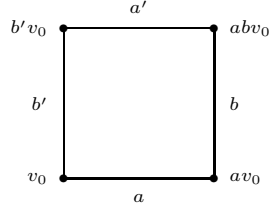
\label{square S}
\hfil
\centerline{
\beginpicture
\setcoordinatesystem units <1cm, 1cm>
\setplotarea  x from -6 to 6,  y from -1 to 1.2
\putrule from -1 -1 to 1 -1
\putrule from -1  1 to 1  1
\putrule from -1 -1 to -1 1
\putrule from  1  -1 to 1  1
\put {$_{b'}$} [r]     at   -1.2 0
\put {$_{b}$} [l]     at    1.2 0
\put {$_{a}$} [t]     at   0 -1.2
\put {$_{a'}$} [b]     at   0  1.2
\put {$_{\bullet}$} at 1 1
\put {$_{\bullet}$} at 1 -1
\put {$_{\bullet}$} at -1 1
\put {$_{\bullet}$} at -1 -1
\put {$_{av_0}$}[l] at 1.2 -1
\put {$_{b'v_0}$}[r] at -1.2 1
\put {$_{v_0}$}[r] at -1.2 -1
\put {$_{abv_0}$}[l] at 1.2  1
\endpicture
}
\hfil
\caption{The square $S$.}
\end{figure}

\begin{remark}
The fact that BM groups are torsion free is an immediate consequence of \cite[Theorem 4.13(2) p.201]{bh}. That much more general result applies to fundamental groups of spaces of non-positive curvature. They are always torsion free.
\end{remark}

\bigskip

\section{A $2$-dimensional subshift associated with a BM group}
\label{subshift}

Identify elements of $\G$ with vertices of $\D$. The set $\cR$ may be identified with the set of $\G$-equivalence classes of oriented basepointed squares (chambers) in $\D$.  We refer to such an equivalence class of squares as a tile. We now construct a
2-dimensional shift system associated with $\Gamma$.

The transition matrices are defined as follows.
If $r=(a,b,b',a'), s=(c,d,d',c')\in \cR$ then define horizontal and vertical transition matrices $M_1, M_2$ as indicated in Figure \ref{trans}: that is $M_j(s,r)=1$ if $r$ and $s$ represent the labels of tiles in $\D$ which lie as shown in Figure \ref{trans}, and  $M_j(s,r)=0$ otherwise. The $mn\times mn$ matrices $M_1, M_2$ are nonzero \zomatrices.

\refstepcounter{picture}
\begin{figure}[htbp]
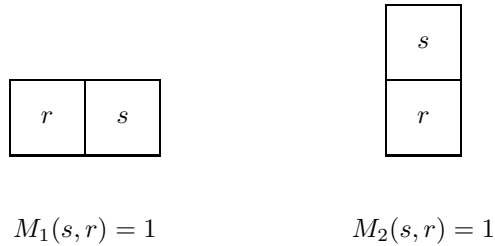
\label{trans}
\hfil
\centerline{
\beginpicture
\setcoordinatesystem units <1cm, 1cm>
\setplotarea  x from -6 to 6,  y from -1 to 2.2
\putrule from -3 0 to -1 0
\putrule from -3 1 to -1 1
\putrule from -3 0 to -3 1
\putrule from -2 0 to  -2 1
\putrule from -1 0 to  -1 1
\putrule from 2 0 to  3 0
\putrule from 2 1 to  3 1
\putrule from 2 2 to  3 2
\putrule from 2 0 to  2 2
\putrule from 3 0 to  3 2
\put {$M_1(s,r)=1$} []     at   -2 -1
\put {$M_2(s,r)=1$} at  2.5 -1
\put {$r$} []     at   -2.5 0.5
\put {$s$} at   -1.5 0.5
\put {$r$} []     at   2.5 0.5
\put {$s$} at   2.5  1.5
\endpicture
}
\hfil
\caption{Definition of the transition matrices.}
\end{figure}

It follows that $M_1(s,r)=1$ if and only if $b=d'$ and $c\ne a^{-1}$. (The condition
$c'\ne {a'}^{-1}$ is redundant, because two adjacent sides of a square uniquely determine it.)  See Figure \ref{trans1}. It follows that each row or column of $M_1$ has precisely $m-1$ nonzero entries. A diagram similar to Figure \ref{trans1} applies to vertical transition matrices, with the result that each row or column of $M_2$ has  precisely $n-1$ nonzero entries.

\refstepcounter{picture}
\begin{figure}[htbp]
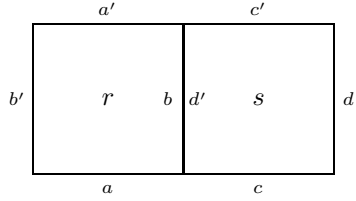
\label{trans1}
\hfil
\centerline{
\beginpicture
\setcoordinatesystem units <2cm, 2cm>
\setplotarea  x from -1.5 to 1.5,  y from -0.2 to 1.2
\putrule from -1 0 to 1 0
\putrule from -1 1 to 1 1
\putrule from -1 0 to -1 1
\putrule from 0 0 to  0 1
\putrule from 1 0 to  1 1
\put {$r$} []     at   -0.5 0.5
\put {$s$} at   0.5 0.5
\put {$_a$} at   -0.5 -0.1
\put {$_c$} at   0.5 -0.1
\put {$_{a'}$} at   -0.5 1.1
\put {$_{c'}$} at   0.5 1.1
\put {$_b$} at   -0.1 0.5
\put {$_d$} at   1.1 0.5
\put {$_{b'}$} at   -1.1 0.5
\put {$_{d'}$} at   0.1 0.5
\endpicture
}
\hfil
\caption{Definition $M_1$.}
\end{figure}

We now use $\cR$ as an alphabet and $M_1, M_2$ as transition matrices to build up two dimensional words as in \cite{rs2}.
Let $[m,n]$ denote $\{m,m+1, \dots , n\}$, where $m \le n$ are integers. If $m,n \in \ZZ^2$, say that $m \le n$ if $m_j \le n_j$ for $j=1,2$, and when $m \le n$, let  $[m,n] = [m_1,n_1] \times [m_2,n_2]$. In $\ZZ^2$, let $0$ denote the zero vector
and let $e_j$ denote the $j^{th}$ standard unit basis vector.
If $m\in \ZZ^2_+=\{m\in \ZZ^2;\ m\ge 0\}$, let
$$W_m = \{ w: [0,m] \to \cR ;\ M_j(w(l+e_j),w(l)) = 1\ \text{whenever}\ l,l+e_j \in [0,m] \}$$ and call the elements of $W_m$ words.
Let $W = \bigcup_{m\in \ZZ_+^2} W_m$.
We say that a word $w\in W_m$ has shape $\s(w)=m$, and we identify $W_0$ with $\cR$ in the natural way via the map $w \mapsto w(0)$. Define the initial and final maps $o: W_m \to \cR$ and $t: W_m \to \cR$ by $o(w) = w(0)$ and $t(w) = w(m)$.
 In order to apply the theory of \cite{rs2} we need to show that the matrices $M_1$,$M_2$ satisfy the following conditions.

\begin{description}
\item[(H0)] Each $M_i$ is a nonzero \zomatrix.

\item[(H1a)] $M_1M_2 =M_2M_1$.

\item[(H1b)] $M_1M_2$ is a \zomatrix.

\item[(H2)] The directed graph with vertices $r \in \cR$
and directed edges $(r,s)$ whenever $M_i(s,r) =1$ for some $i$, is irreducible.

\item[(H3)] For any nonzero $p\in \ZZ^2$, there exists a word $w\in W$ which is not {\em $p$-periodic}, i.e. there exists $l$ so that
$w(l)$ and $w(l+p)$ are both defined but not equal.
\end{description}

\begin{lemma}\label{M_1M_2} The matrices $M_1$, $M_2$ satisfy conditions
{\rm(H0)},{\rm(H1a)}, {\rm(H1b)}, and {\rm(H3)}.
\end{lemma}

\begin{proof}
{\bf (H0):}
By definition $M_1$ and~$M_2$ are $\{0,1\}$-matrices and they are clearly nonzero.

\noindent{\bf (H1a,b):}
Consider the configuration of
Figure~\ref{W1} consisting of chambers lying in some apartment of the building.  Given the tiles $r_o$, $r_1$, and~$r$,
there is exactly one tile~$r_2$ which completes the picture. Therefore if $(M_2M_1)(r,r_o)>0$ then
$(M_1M_2)(r,r_o)=1$.  Likewise, if $(M_1M_2)(r,r_o)>0$, then
$(M_2M_1)(r,r_o)=1$.  Conditions~(H1a) and~(H1b) follow.

\refstepcounter{picture}
\begin{figure}[htbp]\label{W1}
\hfil
\centerline{
\beginpicture
\setcoordinatesystem units <1cm, 1cm>
\setplotarea  x from -1 to 1,  y from -1 to 1
\putrule from -1 -1 to 1 -1
\putrule from -1 0 to 1 0
\putrule from -1 1 to 1 1
\putrule from -1 -1 to  -1 1
\putrule from 0 -1 to  0 1
\putrule from 1 -1 to  1 1
\put {$r_o$}      at   -0.5 -0.5
\put {$r_1$}      at   0.5  -0.5
\put {$r_2$}      at   -0.5  0.5
\put {$r$}      at   0.5 0.5
\endpicture
}
\hfil
\caption{A word of shape $(1,1)$.}
\end{figure}

\noindent{\bf (H3):}
Fix any nonzero $p\in\ZZ^2$. Choose $m\in \ZZ^2_+$ large enough that the rectangle
$[0,m]$ contains a point $l$ and its $p$-translate $l+p$.
We can construct $w\in W_m$ which is not
$p$-periodic, as follows.

Let $w(l)\in \cR$ be chosen arbitrarily. Now for $j=1,2$ there are at least two choices
of $r\in\cR$ such that $M_j(r,w(l))=1$ [respectively $M_j(w(l),r)=1$].
Thus one can begin to extend the domain of definition of $w$ in any one of four directions
so that there are at least two choices of $w(l\pm e_j)$, $j=1,2$.
By induction, one can extend $w$ in many ways to an element of $W_m$, at each step choosing a particular direction for the extension. In order to do this,
first choose arbitrarily a shortest path from $l$ to $l+p$, and then extend step by step along the path.
It is important to note that at each step, $w$ extends uniquely to be defined on a complete rectangle in $\ZZ^2$,
as illustrated in Figure \ref{extension}. In that Figure, we assume that $w$ is defined on
the rectangle $[l,m]$, and then define $w(m+e_2)=r$, where $M_2(r,w(m)) =1$.
By conditions (H1a) and (H1b), there is a unique choice of $w(m-e_1+e_2)$ which is
compatible with the values of $w(m+e_2)$ and $w(m-e_1)$. Continue the process inductively until
$w$ is defined uniquely on the whole rectangle $[l,m+e_2]$.

\refstepcounter{picture}
\begin{figure}[htbp]\label{extension}
\hfil
\centerline{
\beginpicture
\setcoordinatesystem units <1cm, 1cm>
\setplotarea  x from -6 to 6,  y from -1 to 2.2
\putrule from -3 -1 to 2 -1
\putrule from -3 1  to 2  1
\putrule from  -3 -1  to -3 1
\putrule from  2 -1  to 2 2
\setdashes
\putrule from  -3 2 to  2 2
\putrule from  -3 2 to  -3 1
\putrule from  1 2 to  1 1
\put {$[l,m]$}      at   -0.5  0
\put {$_m$} [l]     at   2.2 1
\put {$_l$} [r]     at   -3.2 -1
\put {$_{m+e_2}$}[l]      at   2.2 2
\put {$_{m-e_1}$} [t]     at   1  0.8
\put {$_{m-e_1+e_2}$} [b]     at   1  2.2
\put {$\bullet$} at 2 1
\put {$\bullet$} at 2 2
\put {$\bullet$} at 1 1
\put {$\bullet$} at 1 2
\put {$\bullet$} at -3 -1
\endpicture
}
\hfil
\caption{$w(m+e_2)$ determines $w$ on the rectangle $[l, m+e_2]$.}
\end{figure}

Figure \ref{path} illustrates how $w$ is defined on $[l, l+p]$ (where $p=(5,3)$), by moving along a certain path from $l$ to $l+p$.  The values of $w$ on this path are given by a sequence of tiles.
The values of $w$ up to a certain point on the path determine the values on a rectangle
which contains the corresponding initial segment of the path.  At the end of the process, the values of $w$ on the path have completely determined the values on $[l, l+p]$.
The final extension from $[l, l+p]$ to the complete rectangle $[0,m]$ is done similarly.

\refstepcounter{picture}
\begin{figure}[htbp]\label{path}
\hfil
\beginpicture
\setcoordinatesystem units <0.7 cm, 0.7cm>
\setplotarea  x from -6 to 6,  y from -2 to 2
\putrectangle corners at -3 -2 and -2 -1
\putrectangle corners at -2 -2 and -1 -1
\putrectangle corners at -1 -2 and 0 -1
\putrectangle corners at 0 -2 and 1 -1
\putrectangle corners at 0 -1 and 1  0
\putrectangle corners at 0 0 and 1 1
\putrectangle corners at 1 1 and 2 1
\putrectangle corners at 1 0 and 2 1
\putrectangle corners at 1 1 and 2 2
\putrectangle corners at 2 1 and 3 2
\setdashes
\putrectangle corners at -3 -2 and 3 2
\put{$_{w(l)}$} at -2.5  -1.5
\put{$_{w(p)}$} at 2.5  1.5
\endpicture
\hfil
\caption{Definition of $w$ on the rectangle $[l,l+p]$.}
\end{figure}

At each step, there are at least two different choices for the extension in any direction for which
$w$ is not already defined.
In particular, one can ensure that $w(l+p)\ne w(l)$. Therefore $w$ is not $p$-periodic.

\end{proof}

\begin{lemma}\label{L2}
Consider the directed graph which has a vertex for each $r \in \cR$
and a directed edge from $r$ to $s$ for each $i$ such that $M_i(s,r) =1$.
This graph is irreducible. i.e. Condition {\rm(H2)} holds.
\end{lemma}

\begin{proof}
Given $r_o, r_t \in \cR$ we need to find a directed path starting at $r_o$ and ending at $r_t$.

There are $m-1$ letters $r_1\in \cR$ such that $M_1(r_1,r_o)=1$. For each such $r_1$ there are $n-1$ letters $r\in \cR$ such that $M_2(r,r_1)=1$. Since $M_2M_1$ is a \zomatrix (equivalently the paths $r_o \to r_1 \to r$ are distinct) it follows that the set
$$S_+(r_o)= \{ r \in \cR \st w(0,0)=r_o \text{ and } w(1,1)= r \text{ for some } w\in W_{(1,1)} \}$$
contains $(m-1)(n-1)$ elements. See Figure \ref{W1}.
Similarly the set
$$S_-(r_t)= \{ r \in \cR \st w(0,0)=r \text{ and } w(1,1)= r_t \text{ for some } w\in W_{(1,1)} \}$$
contains $(m-1)(n-1)$ elements.
Since $m,n \ge 4$, we have  $|\cR | = mn < 2(m-1)(n-1)$ and so there exists $r\in S_+(r_o)
\cap S_-(r_t)$. It follows that there is a directed path from $r_o$ (to $r$) to $r_t$, as required.
\end{proof}

Associated with the 2-dimensional shift system constructed above there is a finitely
generated abelian group defined as follows.
The block $mn\times 2mn$ matrix $(I-M_1,I-M_2)$ defines a homomorphism
$\ZZ^{\cR}\oplus \ZZ^{\cR} \to \ZZ^{\cR}$. Define
$C=\coker\begin{smallmatrix}(I-M_1,&I-M_2)\end{smallmatrix}$. Thus $C$ can
be defined as an abelian group, in terms of generators and relations:
\[
C=C(\G)=
\langle r\in\cR ; \;  r=\sum_sM_j(s,r)s,\ j=1,2 \rangle
\]
As we shall see, this group plays an important role
in classifying the $C^*$-algebra $\cA(\G)$ which is studied in the next section.
The next observation will be needed there.

\begin{lemma}\label{Mtranspose} There exists a permutation matrix $P : \ZZ^{\cR} \to \ZZ^{\cR}$
 such that $P^2=I$ and
\[
PM_jP=M_j^t, \qquad j=1,2.
\]
In particular $\coker\begin{smallmatrix}(I-M_1,&I-M_2)\end{smallmatrix} =
\coker\begin{smallmatrix}(I-M_1^t,&I-M_2^t)\end{smallmatrix}$.
\end{lemma}

\begin{proof}
Define $p: \cR \to \cR$ by $p((a,b,b',a'))=({a'}^{-1},{b'}^{-1},b^{-1},a^{-1})$. (This corresponds to a rotation of the square in Figure \ref{geometric square}
through the angle $\pi$.)
Then $M_j(s,r)=1$ if and only if $M_j(p(r),p(s))=1$. That is $M_j(s,r)=M_j(p(r),p(s))$.
Let $P : \ZZ^{\cR} \to \ZZ^{\cR}$ be the corresponding permutation matrix defined by
$Pe_{p(r)}=e_r$, where $\{e_r\st r\in \cR\}$ is the standard basis of $\ZZ^{\cR}$.
\end{proof}

\bigskip

\section{The boundary action.}

 A {\bf sector} in $\D$ is a $\frac{\pi}{2}$-angled sector in some apartment.
Two sectors are  {\bf equivalent} (or parallel) if their intersection contains a sector. See Figure \ref{equivalent sectors}, where the equivalent sectors with base vertices $x$, $x'$  do not necessarily lie in a common apartment, but the shaded subsector is contained in them both.

\refstepcounter{picture}
\begin{figure}[htbp]
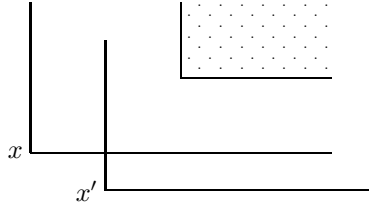
\label{equivalent sectors}
\hfil
\centerline{
\beginpicture
\setcoordinatesystem units <1cm, 1cm>
\setplotarea  x from -6 to 6,  y from -1 to 2.5
\putrule from 0 0 to 4  0
\putrule from 0 0  to 0 2
\putrule from 1 -0.5 to 4.5  -0.5
\putrule from 1 -0.5  to 1 1.5
\putrule from 2  1 to 4  1
\putrule from 2 1    to 2  2
\put {$x$} [r]     at   -0.1 0
\put {$x'$} [r]     at    0.9 -0.5
\setshadegrid span <4pt>
\vshade 2 1 2  4 1 2 /
\endpicture
}
\hfil
\caption{Equivalent sectors containing a common subsector.}
\end{figure}

The boundary $\Om$ of $\D$ is defined to be the set of equivalence classes of sectors in $\D$.  Fix a vertex $x$. For any $\w \in \Om$ there is a unique sector $[x,\w)$ in the
class $\w$ having base vertex $x$, as illustrated in Figure \ref{repsec} \cite[Theorem 9.6]{ron}.

\refstepcounter{picture}
\begin{figure}[htbp]
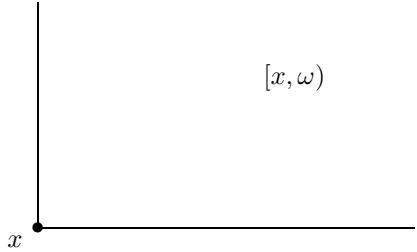
\label{repsec}
\hfil
\centerline{
\beginpicture
\setcoordinatesystem units <1cm, 1cm>
\setplotarea  x from -6 to 6,  y from -1 to 2.2
\putrule from -3 -1 to 2 -1
\putrule from -3 -1 to -3 2
\put {$\bullet$} at -3 -1
\put {$x$}[t,r] at -3.2 -1.1
\put {$[x, \omega)$} [l]     at   0 1
\endpicture
}
\hfil
\caption{A representative sector $[x, \omega)$.}
\end{figure}

$\Om$ is a totally disconnected compact Hausdorff space with a basis for the
topology given by sets of the form
$$
\Om(v) = \left \{ \w \in \Om : [x,\w) \hbox { contains } v \right \}
$$
where $v$ is any fixed vertex of $\D$. It is easy to see that $\Om$ is (non canonically) homeomorphic to $\partial T_1 \times \partial T_2$.

Recall from Section \ref{subshift} that the alphabet $\cR$ is identified with the set of $\G$-equivalence classes of basepointed chambers in $\D$. We refer to such an equivalence class as a tile.
Each tile has a unique representative labelled square based at a fixed vertex of $\D$, where each edge label is a generator of $\G$.

For each vertex $y\in\D$ the convex hull $\conv(x,y)$ is a rectangle $R$ in some apartment.
Associated with the rectangle $R$ there is therefore a unique word $w\in W$ defined by the labellings of the
constituent squares of $R$.
Conversely, by construction of the BM group, every word $w\in W$ arises in this way.
There is thus a natural bijection between the set of rectangles $R$ in $\D$ based at $x$ and  the set of words $w\in W$.
Denote by $\ff(w)$ the basepointed final chamber (square) in the rectangle $R$. Thus $\ff(w)$ has edge labelling corresponding to $t(w)\in \cR$. The square $\ff(w)$ is oriented, with basepoint chosen to be the vertex closest to the origin $x$.
It is worth recalling that the terminology has been set up so that $\cR=W_{(0,0)}$.
That is, tiles are words of shape $(0,0)$.

If $w\in W$, denote by $\Om(w)$ the set of all $\omega\in \Omega$ such that the sector
$[x, \omega)$ contains the rectangle in $\D$ based at $x$, corresponding to the word $w$.
Denote by $\id _{\Om(w)}$ the indicator function of this set.
It is clear from the definition of the topology on $\Omega$ that $\id _{\Om(w)}\in C(\Om)$.

\refstepcounter{picture}
\begin{figure}[htbp]
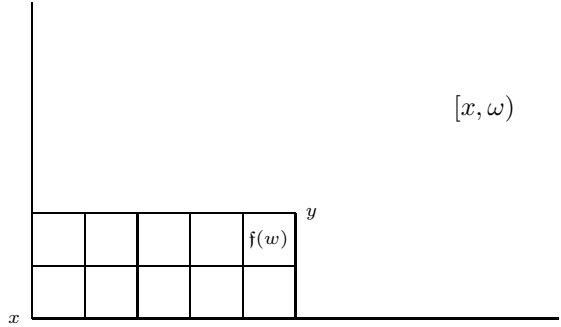
\label{word}
\hfil
\centerline{
\beginpicture
\setcoordinatesystem units <0.7cm, 0.7cm>
\setplotarea  x from -6 to 6,  y from -1 to 5
\putrule from -3 -1 to 7 -1
\putrule from -3 0 to 2 0
\putrule from -3 1  to 2  1
\putrule from  -3 -1  to -3 5
\putrule from  -2 -1  to -2 1
\putrule from  -1 -1  to -1 1
\putrule from  0 -1  to -0 1
\putrule from  1 -1  to 1 1
\putrule from  2 -1  to 2 1
\put {$_y$} [l]     at   2.2 1
\put {$_x$} [r]     at   -3.2 -1
\put {${[x, \omega)}$} [l]     at   5 3
\put {$_{{\mathfrak f}(w)}$}     at   1.5 0.5
\endpicture
}
\hfil
\caption{The rectangle $R=\conv(x,y)$ associated to a word $w\in W_{(4,1)}$,
and the sector $[x, \omega)$ representing a boundary point  $\omega\in\Om(w)$. }
\end{figure}

The group $\G$ acts on $\Om$ and hence on $C(\Om)$ via $\g\mapsto \alpha_{\g}$, where
$\alpha_{\g} f(\w)=f(\g^{-1}\w)$, for $f\in C(\Om)$, $\g\in\G$.
The algebraic crossed product relative to this action is the $*$-algebra $k(\G,C(\Om))$
of functions $\phi : \G \to C(\Om)$ of finite support, with multiplication and involution given by
\[
\phi * \psi (\g_0) = \displaystyle \sum_{\g\in\G}\phi(\g)\alpha_{\g}(\psi(\g^{-1}\g_0))
\quad \text{and} \quad
\phi^*(\g)=\alpha_{\g}(\phi(\g^{-1})^*).
\]

The full crossed product algebra $C(\Omega) \rtimes \G$ is the completion of the algebraic crossed product in an appropriate norm.  There is a natural embedding of $C(\Omega)$ into $C(\Omega) \rtimes \G$ which maps $f\in C(\Omega)$ to the function taking the value $f$ at the identity of $\G$ and $0$ elsewhere.
The identity element $\id$ of $C(\Omega) \rtimes \G$ is then identified with the constant function $\id(\w)=1,\, \w\in\Omega$.
There is a natural unitary representation $\pi : \G \to C(\Omega) \rtimes \G$, where $\pi(\g)$ is the function taking the value $\id$ at $\g$ and $0$ otherwise.
It is convenient to denote $\pi(\g)$ simply by $\g$. Thus a typical element of the dense $*$-algebra $k(\G,C(\Omega))$ can be written as a finite sum $\sum_{\g} f_{\g} \g$, where $f_{\g}\in C(\Omega)$, $\g\in\G$.
The definition of the multiplication implies the covariance relation
\begin{equation}\label{covar}
\alpha_{\g}(f) = \g f \g^{-1} \quad \text{for} \quad f \in C(\Omega), \g \in \Gamma.
\end{equation}

\begin{theorem}\label{main01}
Let $\cA(\G) = C(\Omega) \rtimes \Gamma$. Then $\cA(\G)$ is isomorphic to the rank-2 Cuntz-Krieger algebra $\cA$ associated with the alphabet $\cR$ and transition matrices $M_1, M_2$, as described in \cite{rs2}.
\end{theorem}

The proof of this result is essentially the same as that given in \cite[Section 7]{rs2},
in the case of a group of automorphisms of a building of type $\tilde A_2$.

Here is how the isomorphism is defined.
The $C^*$-algebra $\cA$ is defined as the universal $C^*$-algebra
generated by a family of partial isometries
$\{s_{u,v};\ u,v \in W \ \text{and} \ t(u) = t(v) \}$
satisfying the relations
\begin{subequations}\label{rel1*}
\begin{eqnarray}
{s_{u,v}}^* &=& s_{v,u} \label{rel1a*}\\
s_{u,v}s_{v,w}&=&s_{u,w} \label{rel1b*}\\
s_{u,v}&=&\displaystyle\sum_
{\substack{w\in W;\s(w)=e_j,\\
o(w)=t(u)=t(v)}}
s_{uw,vw} ,\ \text{for} \ 1 \le j \le r
\label{rel1c*}\\
s_{u,u}s_{v,v}&=&0 ,\ \text{for} \ u,v \in  W_0, u \ne v. \label{rel1d*}
\end{eqnarray}
\end{subequations}

We refer to \cite[Section 1]{rs2} for details, in particular for the meaning of the product of words used in
(\ref{rel1c*}).

The isomorphism $\phi: \cA \to C(\Om)\rtimes \G$ is defined as follows.

If $u, v \in W$ with $t(u)=t(v)\in \cR$, let $\g \in \G$ be the unique element such that $\g \ff(v)=\ff(u)$.
The condition $t(u)=t(v)$ means that $\ff(u)$, $\ff(v)$ lie in the same $\Gamma$-orbit, so that $\gamma$ exists.
Moreover $\gamma$ is unique, since $\Gamma$ acts freely on $\D$.
Now define
\begin{equation}\label{isomorphism}
\phi(s_{u,v})=\g \id _{\Om(v)}= \id _{\Om(u)}\g.
\end{equation}

The proof of that $\phi$ is an isomorphism is exactly the same as the corresponding
result for $\tilde A_2$ groups given in \cite[Section 7]{rs2}. Here are the essential details.

The equation (\ref{isomorphism}) does define a $*$-homomorphism of $\cA$ because the operators of the form
$\phi(s_{u,v})$ are easily seen to satisfy the relations (\ref{rel1*}).
Since the algebra $\cA$ is simple \cite[Theorem 5.9]{rs2}, $\phi$ is injective.
Now observe that $\id _{\Om(w)}= \phi(s_{w,w})$.
It follows that the range of $\phi$ contains $C(\Om)$. For the sets $\Om(w)$, $w\in W$,  form a basis for the topology of $\Om$,
and so the linear span of $\{ \id _{\Om(w)} ; w\in W\}$ is dense in $C(\Om)$.
To show that $\phi$ is surjective, it therefore suffices to show that the range of $\phi$ contains $\G$.
It is clearly enough to show that $\phi(\cA)$ contains
the generating set $A\cup B$ for $\G$.

Suppose that $a\in A$. Then
\[
a=a.\id=\sum_{w\in W_{(1,0)}}a\id _{\Om(w)} \in \phi(\cA).
\]
Similarly $B\subset \phi(\cA)$.
\qed

\bigskip

In view of Lemmas \ref{M_1M_2}, \ref{L2}, the following is an immediate consequence of \cite[Proposition 5.11, Theorem 5.9, Corollary 6.4 and Remark 6.5]{rs2}.

\begin{theorem}\label{main1} The $C^*$-algebra $\cA(\G)$ is purely infinite, simple and nuclear.
Moreover it satisfies the Universal Coefficient Theorem.
\end{theorem}

It also follows from \cite{rs2} that $\cA(\G)$ satisfies the U.C.T., hence it is classified by its K-theory, together with the class $[\id]$ of its identity element in $K_0$. It is therefore of interest to determine the K-theory of $\cA(\G)$. The matrices
$(I-M_1,I-M_2)$ and $(I-M_1^t,I-M_2^t)$ define homomorphisms $\ZZ^{\cR}\oplus \ZZ^{\cR} \to \ZZ^{\cR}$.
The K-theory of $\cA(\G)$ can be expressed as follows \cite {rs3},
where $G^{\tors}$ denote the torsion part of a finitely generated abelian group $G$, and
$\rank(G)$ denotes the rank of $G$.
\begin{align*}
\rank(K_0(\cA(\G)))&=\rank(K_1(\cA(\G)))\\
&=
\rank(\coker\begin{smallmatrix}(I-M_1,&I-M_2)\end{smallmatrix})+
\rank(\coker\begin{smallmatrix}(I-M^t_1,&I-M^t_2)\end{smallmatrix}) \\
K_0(\cA(\G))^{\tors} &\cong \coker\begin{smallmatrix}(I-M_1,&I-M_2)\end{smallmatrix}^{\tors} \\
K_1(\cA(\G)^{\tors} &\cong \coker\begin{smallmatrix}(I-M^t_1,&I-M^t_2)\end{smallmatrix}^{\tors}.
\end{align*}

Recall that we defined $C=C(\G)=\coker\begin{smallmatrix}(I-M_1,&I-M_2)\end{smallmatrix}$.
The next result therefore follows from Lemma \ref{Mtranspose}.

\begin{theorem}\label{mainK}
If $\G$ is a BM group then
\begin{equation}\label{Kth}
K_0(\cA(\G))=K_1(\cA(\G))=C\oplus\ZZ^{\rank(C)}.
\end{equation}
\end{theorem}

The identity element in $\cA(\G)$ is denoted by $\id$. As is the case for similar algebras
\cite[Proposition 5.4]{rs3}, \cite{ro2}, the class $[\id]$ has torsion in $K_0(\cA(\G))$.
In the present setup we can be much more precise.  For notational convenience, let $\alpha=\frac{m}{2}, \beta=\frac{n}{2}$
and let $\rho=\gcd(\alpha-1, \beta-1)$.

\begin{proposition}\label{orderof1} Let $[\id]$ be the class in $K_0(\cA(\G))$ of the identity element of $\cA(\G)$, where $\G$ is a BM group. Then $\rho.[\id]=0$.

\noindent (a) If $\rho$ is odd, then the order of $[\id]$ is precisely $\rho$.

\noindent (b) If $\rho$ is even, then the order of $[\id]$ is either
$\rho$ or $\frac{\rho}{2}$.
\end{proposition}

The proof of this result depends upon an examination of explicit projections in $\cA(\G)$.

If $c$ is an oriented basepointed square in $\D$ with base vertex $v_0$, let $\Om(c)$ denote the clopen subset of $\Omega$ consisting of all boundary points with representative sector having initial square $c$ and initial vertex $v_0$. The indicator function $p_c$ of the set $\Om(c)$ is continuous and so lies in $C(\Omega)\subset C(\Omega) \rtimes \G$. See Figure \ref{bdpt}.
The covariance relation (\ref{covar}) implies that the class of $p_c$ in $K_0(\cA(\G))$ depends only on the $\G$-equivalence class of the oriented basepointed square $c$. Recall that we identify such a $\G$-equivalence class with a tile $r\in \cR$. It is therefore appropriate to denote the class of $p_c$ in $K_0(\cA(\G))$ by $[r]$.

\refstepcounter{picture}
\begin{figure}[htbp]
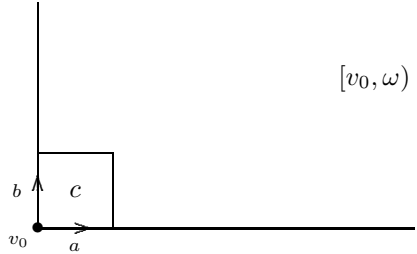
\label{bdpt}
\hfil
\centerline{
\beginpicture
\setcoordinatesystem units <1cm, 1cm>
\setplotarea  x from -6 to 6,  y from -1 to 2.2
\putrule from -3 -1 to 2 -1
\putrule from -3 -1 to -3 2
\putrule from -3 0 to -2 0
\putrule from -2 -1 to  -2 0
\arrow <6pt> [.3,.67] from    -2.5  -1 to   -2.3 -1
\arrow <6pt> [.3,.67] from    -3  -0.5 to   -3 -0.3
\put {$\bullet$} at -3 -1
\put {$[v_0, \omega)$} [l]     at   1 1
\put {$_{a}$} [t]     at   -2.5 -1.2
\put {$_{b}$} [r]     at   -3.2 -0.5
\put {$c$}      at   -2.5 -0.5
\put {$_{v_0}$}[t,r]  at   -3.1 -1.1
\endpicture
}
\hfil
\caption{A sector representing $\omega\in\Om(c)$.}
\end{figure}

Similarly, to each $a\in A$ and $b\in B$ we can associate elements $[a], [b] \in K_0(\cA(\G))$. For example, if $a\in A$, fix a directed edge labelled by the element $a\in A$ and consider the set of all boundary points $\omega$ with representative sector having initial square $c$ containing that edge, as in Figure \ref{bdpt}.  As above, the class in $K_0(\cA(\G))$ of the characteristic function of this set depends only on the label $a$, and may be denoted by $[a]$.  The class $[b]\in K_0(\cA(\G))$ for $b\in B$ is defined similarly.

\bigskip

Recall now the following result.

\begin{lemma}\label{ronan}
{\rm \cite[Lemma 9.4]{ron}}
Given any chamber $c\in\D$ and any sector $S$ in $\D$, there exists a sector $S_1\subset S$ such that $S_1$ and $c$ lie in a common apartment.
\end{lemma}

It follows by considering parallel sectors in an appropriate apartment that
if $e=[v_0,v_1]$ is a directed edge in $\D$ and if $\w\in\Om$, then $\w$ has a representative sector  $S$ that lies relative to $e$ in one of the two positions in Figure \ref{chambersector}, in some apartment containing them both.

\refstepcounter{picture}
\begin{figure}[htbp]
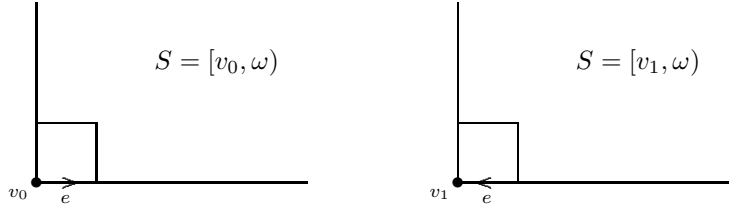
\label{chambersector}
\hfil
\centerline{
\beginpicture
\setcoordinatesystem units <0.8cm, 0.8cm>
\setplotarea  x from -6 to 6,  y from -1 to 2.2
\putrule from -3 -1 to 1.5 -1
\putrule from -3 -1 to -3 2
\putrule from -3 0 to -2 0
\putrule from -2 -1 to  -2 0
\put {$\bullet$} at -3 -1
\put {$S=[v_0,\omega)$}      at   0 1
\put {$_{e}$} [t]     at   -2.5 -1.2
\put {$_{v_0}$}[t,r]  at   -3.1 -1.1
\arrow <6pt> [.3,.67] from    -2.5  -1 to   -2.3 -1
\setcoordinatesystem units <0.8cm, 0.8cm> point at -7 0
\setplotarea  x from -6 to 6,  y from -1 to 2.2
\putrule from -3 -1 to 1.5 -1
\putrule from -3 -1 to -3 2
\putrule from -3 0 to -2 0
\putrule from -2 -1 to  -2 0
\put {$\bullet$} at -3 -1
\put {$S=[v_1, \omega)$}      at   0 1
\put {$_{e}$} [t]     at   -2.5 -1.2
\put {$_{v_1}$}[t,r]  at   -3.1 -1.1
\arrow <6pt> [.3,.67] from    -2.5 -1   to    -2.7 -1
\endpicture
}
\hfil
\caption{Relative positions of a directed edge and a representative sector.}
\end{figure}

Let $p_e$ denote the characteristic function of the set of points $\w\in \Om$
such that $e$ is contained in $[v_0, \w)$, as in the left hand diagram.
Let $\overline p_e$ denote the characteristic function of the set of points $\w\in \Om$
such that $e$ is contained in $[v_1, \w)$, as in the right hand diagram.
 It follows that $p_e, \overline p_e$ are idempotents in $\cA(\G)$ and
$\id = p_e + \overline p_e$.
If the edge $e$ has label $a\in A$ [respectively $b\in B$] then in
$K_0(\cA(\G))$, $[p_e]= [a]\, \text{and}\, [\overline p_e]=[a^{-1}]$
[respectively $[p_e]= [b]\, \text{and}\, [\overline p_e]=[b^{-1}]$].

We therefore obtain

\begin{equation}\label{relationsp}
\begin{split}
[\id]& = [a] + [a^{-1}], \qquad a\in A\\
& = [b] + [b^{-1}], \qquad b\in B.
\end{split}
\end{equation}

The relations (\ref{relationsp}) imply that

\begin{equation*}
\begin{split}
\alpha[\id]& = \displaystyle\sum_{a\in A}[a],\\
\beta[\id]& = \displaystyle\sum_{a\in B}[b].
\end{split}
\end{equation*}

Also, each boundary point $\w$ has a unique representative sector
based at a fixed vertex $v_0$ with initial edges as in Figure \ref{bdpt}.
It follows that

\begin{equation*}
[\id] = \displaystyle\sum_{a\in A}[a] =
\displaystyle\sum_{b\in B}[b].
\end{equation*}

Therefore

\begin{equation*}
[\id] = \alpha[\id]=\beta[\id],
\end{equation*}

from which it follows that $\rho.[\id]=0$, thus proving the first assertion
 in Proposition \ref{orderof1}.

In order to obtain a lower bound for the order of the class $[\id]$ in $K_0(\cA(\G))$ we need to use the fact (proved in \cite{rs3}) that the map $r\mapsto [r]$ is a monomorphism from the abstract group
\[
C=
\langle r\in\cR ; r=\sum_sM_j(s,r)s,\ j=1,2 \rangle
\]
onto a direct summand of $K_0(\cA(\G))$. The class $[\id]$ is the image of the element
$e=\displaystyle\sum_{r\in\cR} r$ under this map.
Moreover each column of the matrix $M_1$ [respectively $M_2$] has  $(m-1)$ [respectively $(n-1)$] nonzero terms. (See Section \ref{subshift}.)

Let $k=2\rho=\gcd(m-2,n-2)$, and define a map $\phi : C \to \ZZ/k\ZZ$ by
$\phi(r)= 1 + \ZZ/k\ZZ$.
The map $\phi$ is well defined since each relation in the presentation of $C$ expresses a generator $r$ as a sum of $(m-1)$ or $(n-1)$ other generators.  Also $\phi(e)=mn+\ZZ/k\ZZ = 4 + \ZZ/k\ZZ$, since
$(m-2)(n-2)=mn-2(m-2)-2(n-2)-4$.  There are now two cases to consider.

\noindent (a)  Suppose that $\rho$ is odd. Then $\phi(e)$ has order $\rho$ in $\ZZ/k\ZZ$.
Therefore the order of $e$ in $C$ is divisible by $\rho$ and hence equal to $\rho$.

\noindent (b)  Suppose that $\rho$ is even. Then $\phi(e)$ has order $\frac{\rho}{2}$ in $\ZZ/k\ZZ$. Therefore the order of $e$ in $C$ is divisible by $\frac{\rho}{2}$.
\qed

\begin{remark}\label{conjectures}
It is tempting to conjecture that the order of the class $[\id]$ is always precisely $\rho$. As we shall see below, there is some supporting evidence for this.
There is also computational evidence that $\rank(C)=\rank(H_2(\G))$,
that is $\rank(K_0(\cA(\G))) =2\rank(H_2(\G))$.
\end{remark}

\begin{remark}\label{tpuct}
Recall from Theorem \ref{main1} that $\cA(\G)$ is a p.i.s.u.n. $C^*$-algebra
satisfying the UCT. Furthermore, by Theorem \ref{mainK},
$$K_0(\cA(\G))=K_1(\cA(\G))=\ZZ^{2n}\oplus T$$
where $T$ is a finite abelian group.
It follows from \cite[Proposition 7.3]{ro1} that $\cA(\G)$ is stably isomorphic to $\cA_1 \otimes \cA_2$,
where $\cA_1$, $\cA_2$ are simple rank one Cuntz-Krieger algebras and $K_0(\cA_2)=K_1(\cA_2)=\ZZ$.
\end{remark}

\bigskip

\section{Examples.}\label{exs}

In this section we consider some examples of BM groups $\G$ and the results of
the computations for the group $C=C(\G)$.
It is useful to relate our results to the Euler-Poincar\'e characteristic $\chi(\G)$, which
is the alternating sum of the ranks of the groups $H_i(\G)$.
The finite cell complex $\G\backslash\D$ is a $K(\G,1)$ space and $\G$ has homological dimension at most two, so that $H_2(\G)$ is free abelian and $H_i(\G)=0$ for $i>2$ . It follows that  $\chi(\G)$ coincides with the usual Euler-Poincar\'e characteristic of the cell complex ${\mathcal Y}=\G\backslash\D$.
Explicitly, $\chi(\G)=(\alpha -1)(\beta-1)$, where $m=2\alpha$ and $n=2\beta$.

\begin{example}
Suppose that $\G$ is a direct product of free groups of ranks $\alpha$ and $\beta$, acting on a product $\D$  of homogeneous trees of degrees $m=2\alpha$ and $n=2\beta$ respectively. Then direct computation shows that
\begin{equation*}
K_0(\cA(\G)) = \ZZ^{2\alpha\beta}\oplus(\ZZ/(\beta-1)\ZZ)^\alpha\oplus(\ZZ/(\alpha-1)\ZZ)^\beta
\oplus{\ZZ/\rho\ZZ}
\end{equation*}
where $\rho=\gcd(\alpha-1,\beta-1)$ is the order of the class $[\id]$.
For this group, $H_2(\G)=\ZZ^{\alpha\beta}$, and the conjectures of
Remark \ref{conjectures} are verified.

In this example  $\cA(\G)$ is actually isomorphic in a natural way to a tensor product
of Cuntz-Krieger algebras. In fact, by a result of J. Spielberg \cite[Section 1]{ro1}, the action of a free group $\FF_{\alpha}$ on the boundary of
its Cayley graph gives rise a Cuntz-Krieger algebra $\cA(\FF_{\alpha})$. It is easy to check that
$\cA(\G)\cong \cA(\FF_{\alpha})\otimes \cA(\FF_{\beta})$.
(Compare with Remark \ref{tpuct}.)
The formula for $K_0(\cA(\G))$ can thus also be verified using the K\"unneth Theorem for tensor products.
\end{example}

\begin{example}
Consider some specific examples studied in \cite[Section 3]{moz}.  If $p,l \equiv 1 \pmod 4$ are two distinct primes, Mozes constructs a lattice subgroup $\G_{p,l}$ of $G=PGL_2(\QQ_p) \times
PGL_2(\QQ_l)$.  The building $\D$ of $G$ is a product of two homogeneous trees
$T_1$, $T_2$ of degrees $(p+1)$ and $(l+1)$ respectively. The group $\G_{p,l}$ is a BM group which acts freely and transitively on the vertex set of $\D$, but $\G_{p,l}$
 is {\bf not} a product of free groups. In fact $\G_{p,l}$ is an irreducible lattice in $G$.

Here is how $\G_{p,l}$ is constructed \cite{moz}. Let $\HH(\ZZ)=\{\alpha=a_0+a_1i+a_2j+a_3k ; a_j\in \ZZ\}$, the ring of integer quaternions. Let $i_p$ be a square root of $-1$ in $\QQ_p$ and define
$$\psi : \HH(\ZZ) \to PGL_2(\QQ_p) \times PGL_2(\QQ_l)$$
 by
$$\psi(a_0+a_1i+a_2j+a_3k)=\left(
\begin{bmatrix}
a_0+a_1i_p & a_2+a_3i_p \\
-a_2+a_3i_p  & a_0-a_1i_p \\
\end{bmatrix},
\begin{bmatrix}
a_0+a_1i_l & a_2+a_3i_l \\
-a_2+a_3i_l  & a_0-a_1i_l \\
\end{bmatrix}
\right).
$$

Let $\tilde\G_{p,l}=\{\alpha=a_0+a_1i+a_2j+a_3k\in \HH(\ZZ) ; a_0\equiv 1 \pmod 2,
a_j\equiv 0 \pmod 2, j=1,2,3, |\alpha|^2=p^rl^s\}$.
Then $\G_{p,l}=\psi(\tilde\G_{p,l}$) is a torsion free cocompact lattice in $G$.
Let
\begin{align*}
A&=\{a=a_0+a_1i+a_2j+a_3k\in \tilde\G_{p,l} ;\;  a_0>0, |a|^2=p\}, \\
B&=\{b=b_0+b_1i+b_2j+b_3k\in \tilde\G_{p,l} ;\;  b_0>0, |b|^2=l\}.
\end{align*}

Then $A$ contains $p+1$ elements and $B$ contains $l+1$ elements.
The images $\underline A$, $\underline B$ of $A, B$ in $\G_{p,l}$ generate free groups $\G_p$, $\G_l$ of orders $\frac{p+1}{2}$, $\frac{l+1}{2}$ respectively and $\G_{p,l}$ itself is generated by $\underline A \cup \underline B$.
The product $T_1\times T_2$ is the Cayley graph of $\G_{p,l}$ relative to this set of generators.

The group $\G_{p,l}$ is a BM group and $\rho=\frac{1}{2}\gcd(p-1,l-1)$ is even.  Explicit computations, using the formula (\ref{Kth}) and the MAGMA computer algebra package, show that the order of $[\id]$ is $\rho$ in each of the 28 groups $\G_{p,l}$ where $p,l \equiv 1 \pmod 4$ are two distinct primes $\le 61$.

The normal subgroup theorem \cite[Theorem 4.1]{bm2} can be applied to $\G_{p,l}$, if the Legendre symbol
$\left(\begin{smallmatrix}p \\ l\end{smallmatrix}\right)=1$. For, using the notation of \cite[Section 2.4]{bm2} and \cite[Remarks following Proposition 1.8.1]{bm3},
the group $H_p^{\infty}=PSL_2( \QQ_p)$ has finite index in $H_p=\overline {pr_p(\Gamma_{p,l})}$.
Moreover $H_p$ is locally $\infty$ transitive. Thus the hypotheses of [BM2: Theorem 4.1] are satisfied.

Applying \cite[Theorem 4.1]{bm2} to the commutator subgroup
$[\G_{p,l}, \G_{p,l}]$ of  $\G_{p,l}$ shows that the abelianization
$H_1(\G_{p,l})=\G_{p,l}/[\G_{p,l}, \G_{p,l}]$ is finite.

The Euler-Poincar\'e characteristic of $\G_{p,l}$ is
$\chi=\chi(\G_{p,l})=\frac{(p-1)(l-1)}{4}$.
Thus $\chi(\G_{p,l})= \rank H_0(\G_{p,l}) + \rank H_2(\G_{p,l})$
and so $H_2(\G_{p,l})=\ZZ^{\chi -1}$.
Explicit computations for the same range of values of $p, l$ as above shows that
the second conjecture of Remark \ref{conjectures} is also verified in these cases.

Another experimental observation is the following. Checking through a large number of
values of the pairs
of primes $(p,l)$ congruent to 1 mod 4, one sees that the abelianization of $\G_{p,l}$
seems to depend only the greatest common divisor
$r=\gcd((p-1)/4,(l-1)/4,6)$. In fact we conjecture that

\begin{equation*}
H_1(\G_{p,l}) =
\begin{cases}
    \ZZ/2\ZZ \oplus (\ZZ/4\ZZ)^3 &  \text{if\; $r=1$},\\
    (\ZZ/2\ZZ)^3 \oplus (\ZZ/8\ZZ)^2 &  \text{if\; $r=2$},\\
  \ZZ/2\ZZ \oplus \ZZ/3\ZZ \oplus (\ZZ/4\ZZ)^3  &  \text{if\; $r=3$},\\
  (\ZZ/2\ZZ)^3 \oplus \ZZ/3\ZZ \oplus (\ZZ/8\ZZ)^2 &  \text{if\; $r=6$}.
\end{cases}
\end{equation*}
\end{example}

The validity of this formula was checked for all values of the pair
$(p,l)$ up to $(73,97)$, and for several other values.

\bigskip

\section{BM groups degrees $m=n=4$}\label{bm2by2}

A BM group acts on a product of homogeneous trees $T_1\times T_2$, where
$T_1$ has degree $m=2\alpha\ge 4$ and $T_2$ has degree $n=2\beta\ge 4$.
We now examine the simplest case $m=n=4$.

Recall that a BM square is a set of four distinct tuples
$$\{(a,b,b',a'), (a^{-1},b',b,{a'}^{-1}) ,({a'}^{-1},{b'}^{-1},b^{-1},a^{-1}), (a',b^{-1},{b'}^{-1},a)\}.$$
contained in $A\times B\times B\times A$,  as in Definition \ref{def BM group}(i),(ii).
In the presentation (\ref{presentation}), a BM group has $4\alpha\beta$ generators
and a set of relations which we call a set of BM relations.
Possible sets of BM relations correspond to those unions of $\alpha\beta$ BM
squares whose images under a projection chosen from Definition \ref{def BM group}(iii)
are disjoint.

Choose once and for all such a projection, say $(a,b,b',a') \mapsto  (a,b)$.
Consider the graph ${\mathfrak G}$ whose vertices are the BM squares and whose edges join BM squares
having disjoint images under the chosen projection.
Possible sets of BM relations correspond to cliques (complete subgraphs) of size $\alpha\beta$ in the graph
${\mathfrak G}$.
The computer algebra package MAGMA contains a dynamic programming algorithm due to \cite{wendym} for finding all cliques of a given size.

For $\alpha=\beta=2$ there are 541 cliques in ${\mathfrak G}$ of size $4$. Say that two of the corresponding presentations are equivalent if one may be obtained from the other by permuting the generators.
There are precisely 52 equivalence classes of presentations.
In the tables below we list 52 representative presentations, together with the structure of the groups $H_1(\Gamma)$ and $C(\Gamma)$.
For BM groups  $\G$ with semi-degrees $\alpha=\beta=2$, we have $\chi(\G)=1$. Thus
$\rank H_1(\Gamma) = \rank H_2(\Gamma)$.
Therefore the tables show that the conjectures of
Remark \ref{conjectures} are verified for all such groups.

In the tables, each of the 52 groups arising from these presentations is assigned a name $\Gamma= 2\times2.j$.
Using an algorithm due to C. C. Sims \cite{ccs}, which is implemented in the MAGMA package,
we have found the indices and lengths of all conjugacy classes of subgroups with index at
most 8 in each of these 52 groups. From these data we have checked that the only possible isomorphisms between $2\times2.j$ and $2\times2.k$ occur if both $j$ and $k$ lie in one of the following sets:

{\small
$\{1 ,17 \}$, $\{3 ,19 \}$, $\{4 ,30 \}$, $\{5 ,10 \}$,
$\{7 ,21 \}$, $\{26 ,46 \}$,
$\{27 ,29, 45 \}$, $\{28 ,43, 47 \}$, $\{42 ,44 \}$.
}

\noindent Thus, amongst
the 52 groups $2\times2.j$, there are at least 41 non-isomorphic groups.

The quotient complexes $\cY=(T_1\times T_2)/\G$ corresponding to the groups
$\G=2\times2.j$ with $j$ in any of these sets are not homeomorphic, as can be seen by considering the fundamental groups of one point deletion subspaces of $\cY$. One therefore obtains 52
pairwise non-homeomorphic complexes $\cY$.
Of course, this does not imply that the groups $\G$ are non-isomorphic.
In fact, we have discovered Tietze transformations showing that,
if $\{j,k\}$ is contained in any one of the listed sets, except for $\{4,30 \}$
or  $\{5 ,10 \}$, then the corresponding
groups are isomorphic.
There are thus at most 43 non-isomorphic groups.
We do not know whether  the groups $2\times2.j$ and $2\times2.k$ are isomorphic
in the two cases $\{j,k\}=\{4,30 \}$ and  $\{j,k\}=\{5 ,10 \}$.
Note that the shift system group $C(\Gamma)$ depends only on $\G$ and not on the presentation of $\G$.

The following notation is used in the tables.
The group $\Gamma$ is generated by $A\cup B$ where
$A=\{1,-1,2,-2\}$ and $B=\{3,-3,4,-4\}$.
The fixed point free involution on each of these sets is given by
$x\mapsto -x$. Thus $x^{-1}$ is denoted $-x$.
Moreover the relations are written in the form $aba^{\prime}b^{\prime}=1$
rather than the form  $ab=b^{\prime}a^{\prime}$ used in
Definition \ref{def BM group}.
For example, the first relation for the presentation of group $2\times 2.01$ is
$(1)(3)(1)(3)^{-1}=1$.
The tables are simplified by use of the abbreviations:
\begin{itemize}
\item[] $m\, [a,b,...]$ means $\ZZ^m\oplus\ZZ/a\ZZ \oplus \ZZ/b\ZZ \oplus \dots$;
\item[] $(j)a$ means $a,a,...,a$ \quad  with $j$ repetitions.
\end{itemize}

\bigskip

\bigskip

{\tiny
\[
\begin{tabular}{|l|l|l|l|}
\hline
Group & Presentation
&$H_1\left( \Gamma \right) $
&$C\left( \Gamma \right) $
\\\hline
2x2.01
&$\begin{matrix}
+1&+3&+1&-3\qquad&+1&+4&+1&-4
\\+2&+3&+2&-3\qquad&+2&+4&+2&-4
\end{matrix}$
&2$\left[ \left( 2\right) 2\right] $
&2$\left[ \left( 2\right) 4\right] $
\\\hline
2x2.02
&$\begin{matrix}
+1&+3&+1&-3\qquad&+1&+4&+1&-4
\\+2&+3&+2&-3\qquad&+2&+4&-2&+4
\end{matrix}$
&1$\left[ \left( 3\right) 2\right] $
&1$\left[ \left( 3\right) 4\right] $
\\\hline
2x2.03
&$\begin{matrix}
+1&+3&+1&-3\qquad&+1&+4&+1&-4
\\+2&+3&+2&-3\qquad&+2&+4&-2&-4
\end{matrix}$
&2$\left[ \left( 2\right) 2\right] $
&2$\left[ \left( 2\right) 4\right] $
\\\hline
2x2.04
&$\begin{matrix}
+1&+3&+1&-3\qquad&+1&+4&+1&-4
\\+2&+3&+2&+4\qquad&+2&-3&+2&-4
\end{matrix}$
&1$\left[ 2,4\right] $
&1$\left[ 2,4,8\right] $
\\\hline
2x2.05
&$\begin{matrix}
+1&+3&+1&-3\qquad&+1&+4&+1&-4
\\+2&+3&-2&+3\qquad&+2&+4&-2&+4
\end{matrix}$
&1$\left[ \left( 3\right) 2\right] $
&1$\left[ \left( 3\right) 4\right] $
\\\hline
2x2.06
&$\begin{matrix}
+1&+3&+1&-3\qquad&+1&+4&+1&-4
\\+2&+3&-2&+3\qquad&+2&+4&-2&-4
\end{matrix}$
&2$\left[ \left( 2\right) 2\right] $
&2$\left[ \left( 2\right) 4\right] $
\\\hline
2x2.07
&$\begin{matrix}
+1&+3&+1&-3\qquad&+1&+4&+1&-4
\\+2&+3&-2&-3\qquad&+2&+4&-2&-4
\end{matrix}$
&3$\left[ 2\right] $
&3$\left[ 4\right] $
\\\hline
2x2.08
&$\begin{matrix}
+1&+3&+1&-3\qquad&+1&+4&+1&-4
\\+2&+3&-2&+4\qquad&+2&+4&-2&+3
\end{matrix}$
&2$\left[ 2\right] $
&2$\left[ 2,4\right] $
\\\hline
2x2.09
&$\begin{matrix}
+1&+3&+1&-3\qquad&+1&+4&+1&-4
\\+2&+3&-2&+4\qquad&+2&+4&-2&-3
\end{matrix}$
&1$\left[ \left( 2\right) 2\right] $
&1$\left[ \left( 2\right) 2,4\right] $
\\\hline
2x2.10
&$\begin{matrix}
+1&+3&+1&-3\qquad&+1&+4&-1&+4
\\+2&+3&+2&-3\qquad&+2&+4&-2&-4
\end{matrix}$
&1$\left[ \left( 3\right) 2\right] $
&1$\left[ \left( 3\right) 4\right] $
\\\hline
2x2.11
&$\begin{matrix}
+1&+3&+1&-3\qquad&+1&+4&-1&+4
\\+2&+3&+2&+4\qquad&+2&-3&+2&-4
\end{matrix}$
&0$\left[ \left( 2\right) 2,4\right] $
&0$\left[ \left( 2\right) 4,8\right] $
\\\hline
2x2.12
&$\begin{matrix}
+1&+3&+1&-3\qquad&+1&+4&-1&+4
\\+2&+3&-2&+3\qquad&+2&+4&+2&-4
\end{matrix}$
&0$\left[ \left( 4\right) 2\right] $
&0$\left[ \left( 2\right) 2,\left( 3\right) 4\right] $
\\\hline
2x2.13
&$\begin{matrix}
+1&+3&+1&-3\qquad&+1&+4&-1&+4
\\+2&+3&-2&+3\qquad&+2&+4&-2&-4
\end{matrix}$
&1$\left[ \left( 3\right) 2\right] $
&1$\left[ \left( 3\right) 4\right] $
\\\hline
2x2.14
&$\begin{matrix}
+1&+3&+1&-3\qquad&+1&+4&-1&+4
\\+2&+3&-2&-3\qquad&+2&+4&-2&-4
\end{matrix}$
&2$\left[ \left( 2\right) 2\right] $
&2$\left[ \left( 2\right) 4\right] $
\\\hline
2x2.15
&$\begin{matrix}
+1&+3&+1&-3\qquad&+1&+4&-1&+4
\\+2&+3&-2&+4\qquad&+2&+4&-2&+3
\end{matrix}$
&1$\left[ \left( 2\right) 2\right] $
&1$\left[ \left( 2\right) 4\right] $
\\\hline
2x2.16
&$\begin{matrix}
+1&+3&+1&-3\qquad&+1&+4&-1&+4
\\+2&+3&-2&+4\qquad&+2&+4&-2&-3
\end{matrix}$
&1$\left[ \left( 2\right) 2\right] $
&1$\left[ \left( 2\right) 4\right] $
\\\hline
2x2.17
&$\begin{matrix}
+1&+3&+1&-3\qquad&+1&+4&-1&-4
\\+2&+3&+2&-3\qquad&+2&+4&-2&-4
\end{matrix}$
&2$\left[ \left( 2\right) 2\right] $
&2$\left[ \left( 2\right) 4\right] $
\\\hline
2x2.18
&$\begin{matrix}
+1&+3&+1&-3\qquad&+1&+4&-1&-4
\\+2&+3&+2&+4\qquad&+2&-3&+2&-4
\end{matrix}$
&1$\left[ 2,4\right] $
&1$\left[ 2,4,8\right] $
\\\hline
2x2.19
&$\begin{matrix}
+1&+3&+1&-3\qquad&+1&+4&-1&-4
\\+2&+3&-2&-3\qquad&+2&+4&+2&-4
\end{matrix}$
&2$\left[ \left( 2\right) 2\right] $
&2$\left[ \left( 2\right) 4\right] $
\\\hline
2x2.20
&$\begin{matrix}
+1&+3&+1&-3\qquad&+1&+4&-1&-4
\\+2&+3&-2&-3\qquad&+2&+4&-2&+4
\end{matrix}$
&2$\left[ \left( 2\right) 2\right] $
&2$\left[ \left( 2\right) 4\right] $
\\\hline
\end{tabular}
\]

\newpage
\[
\begin{tabular}{|l|l|l|l|}
\hline
Group & Presentation
&$H_1\left( \Gamma \right) $
&$C\left( \Gamma \right) $
\\\hline
2x2.21
&$\begin{matrix}
+1&+3&+1&-3\qquad&+1&+4&-1&-4
\\+2&+3&-2&-3\qquad&+2&+4&-2&-4
\end{matrix}$
&3$\left[ 2\right] $
&3$\left[ 4\right] $
\\\hline
2x2.22
&$\begin{matrix}
+1&+3&+1&-3\qquad&+1&+4&-1&-4
\\+2&+3&-2&+4\qquad&+2&+4&-2&+3
\end{matrix}$
&2$\left[ 2\right] $
&2$\left[ 2,4\right] $
\\\hline
2x2.23
&$\begin{matrix}
+1&+3&+1&-3\qquad&+1&+4&-1&-4
\\+2&+3&-2&+4\qquad&+2&+4&-2&-3
\end{matrix}$
&1$\left[ \left( 2\right) 2\right] $
&1$\left[ \left( 2\right) 2,4\right] $
\\\hline
2x2.24
&$\begin{matrix}
+1&+3&+1&-3\qquad&+1&+4&+2&+4
\\+1&-4&+2&-4\qquad&+2&+3&+2&-3
\end{matrix}$
&1$\left[ 2,4\right] $
&1$\left[ \left( 2\right) 2,8\right] $
\\\hline
2x2.25
&$\begin{matrix}
+1&+3&+1&-3\qquad&+1&+4&+2&+4
\\+1&-4&+2&-4\qquad&+2&+3&-2&-3
\end{matrix}$
&1$\left[ 2,4\right] $
&1$\left[ \left( 2\right) 2,8\right] $
\\\hline
2x2.26
&$\begin{matrix}
+1&+3&+1&-3\qquad&+1&+4&+2&-4
\\+1&-4&+2&+4\qquad&+2&+3&+2&-3
\end{matrix}$
&2$\left[ 2\right] $
&2$\left[ \left( 2\right) 2\right] $
\\\hline
2x2.27
&$\begin{matrix}
+1&+3&+1&-3\qquad&+1&+4&+2&-4
\\+1&-4&+2&+4\qquad&+2&+3&-2&-3
\end{matrix}$
&2$\left[ 2\right] $
&2$\left[ \left( 2\right) 2\right] $
\\\hline
2x2.28
&$\begin{matrix}
+1&+3&+1&-3\qquad&+1&+4&+2&-4
\\+1&-4&-2&+4\qquad&+2&+3&+2&-3
\end{matrix}$
&2$\left[ 2\right] $
&2$\left[ \left( 2\right) 2\right] $
\\\hline
2x2.29
&$\begin{matrix}
+1&+3&+1&-3\qquad&+1&+4&+2&-4
\\+1&-4&-2&+4\qquad&+2&+3&-2&-3
\end{matrix}$
&2$\left[ 2\right] $
&2$\left[ \left( 2\right) 2\right] $
\\\hline
2x2.30
&$\begin{matrix}
+1&+3&+1&+4\qquad&+1&-3&+1&-4
\\+2&+3&+2&+4\qquad&+2&-3&+2&-4
\end{matrix}$
&1$\left[ 2,4\right] $
&1$\left[ 2,4,8\right] $
\\\hline
2x2.31
&$\begin{matrix}
+1&+3&+1&+4\qquad&+1&-3&+1&-4
\\+2&+3&+2&-4\qquad&+2&-3&+2&+4
\end{matrix}$
&0$\left[ 2,\left( 2\right) 4\right] $
&0$\left[ \left( 2\right) 2,\left( 2\right) 8\right] $
\\\hline
2x2.32
&$\begin{matrix}
+1&+3&+1&+4\qquad&+1&-3&+1&-4
\\+2&+3&-2&-3\qquad&+2&+4&-2&-4
\end{matrix}$
&2$\left[ 4\right] $
&2$\left[ 2,8\right] $
\\\hline
2x2.33
&$\begin{matrix}
+1&+3&+1&+4\qquad&+1&-3&+1&-4
\\+2&+3&-2&+4\qquad&+2&+4&-2&+3
\end{matrix}$
&2$\left[ 2\right] $
&2$\left[ 2,4\right] $
\\\hline
2x2.34
&$\begin{matrix}
+1&+3&+1&+4\qquad&+1&-3&+1&-4
\\+2&+3&-2&+4\qquad&+2&+4&-2&-3
\end{matrix}$
&1$\left[ \left( 2\right) 2\right] $
&1$\left[ \left( 2\right) 2,4\right] $
\\\hline
2x2.35
&$\begin{matrix}
+1&+3&+1&+4\qquad&+1&-3&+1&-4
\\+2&+3&-2&-4\qquad&+2&+4&-2&-3
\end{matrix}$
&1$\left[ 2,4\right] $
&1$\left[ \left( 2\right) 2,8\right] $
\\\hline
2x2.36
&$\begin{matrix}
+1&+3&+1&+4\qquad&+1&-3&+2&-3
\\+1&-4&+2&-4\qquad&+2&+3&+2&+4
\end{matrix}$
&0$\left[ 4,8\right] $
&0$\left[ \left( 2\right) 2,4,8\right] $
\\\hline
2x2.37
&$\begin{matrix}
+1&+3&+1&+4\qquad&+1&-3&+2&-3
\\+1&-4&-2&-4\qquad&+2&+3&+2&-4
\end{matrix}$
&0$\left[ \left( 2\right) 2,\left( 2\right) 3\right] $
&0$\left[ \left( 2\right) 4,\left( 2\right) 3\right] $
\\\hline
2x2.38
&$\begin{matrix}
+1&+3&+1&+4\qquad&+1&-3&+2&-4
\\+1&-4&+2&-3\qquad&+2&+3&+2&+4
\end{matrix}$
&1$\left[ 8\right] $
&1$\left[ \left( 2\right) 2,8\right] $
\\\hline
2x2.39
&$\begin{matrix}
+1&+3&+1&+4\qquad&+1&-3&+2&-4
\\+1&-4&-2&-3\qquad&+2&+3&-2&+4
\end{matrix}$
&1$\left[ 2\right] $
&1$\left[ \left( 3\right) 2\right] $
\\\hline
2x2.40
&$\begin{matrix}
+1&+3&+1&+4\qquad&+1&-3&+2&-4
\\+1&-4&-2&-3\qquad&+2&+3&-2&-4
\end{matrix}$
&0$\left[ \left( 2\right) 2,3\right] $
&0$\left[ \left( 2\right) 2,4,3\right] $
\\\hline
2x2.41
&$\begin{matrix}
+1&+3&-1&-3\qquad&+1&+4&-1&-4
\\+2&+3&-2&-3\qquad&+2&+4&-2&-4
\end{matrix}$
&4$\left[ \right] $
&4$\left[ \right] $
\\\hline
2x2.42
&$\begin{matrix}
+1&+3&-1&-3\qquad&+1&+4&-1&-4
\\+2&+3&-2&+4\qquad&+2&+4&-2&+3
\end{matrix}$
&3$\left[ \right] $
&3$\left[ 2\right] $
\\\hline
2x2.43
&$\begin{matrix}
+1&+3&-1&-3\qquad&+1&+4&-1&-4
\\+2&+3&-2&+4\qquad&+2&+4&-2&-3
\end{matrix}$
&2$\left[ 2\right] $
&2$\left[ \left( 2\right) 2\right] $
\\\hline
2x2.44
&$\begin{matrix}
+1&+3&-1&+4\qquad&+1&+4&-1&+3
\\+2&+3&-2&+4\qquad&+2&+4&-2&+3
\end{matrix}$
&3$\left[ \right] $
&3$\left[ 2\right] $
\\\hline
2x2.45
&$\begin{matrix}
+1&+3&-1&+4\qquad&+1&+4&-1&+3
\\+2&+3&-2&+4\qquad&+2&+4&-2&-3
\end{matrix}$
&2$\left[ 2\right] $
&2$\left[ \left( 2\right) 2\right] $
\\\hline
2x2.46
&$\begin{matrix}
+1&+3&-1&+4\qquad&+1&+4&-1&+3
\\+2&+3&-2&-4\qquad&+2&+4&-2&-3
\end{matrix}$
&2$\left[ 2\right] $
&2$\left[ \left( 2\right) 2\right] $
\\\hline
2x2.47
&$\begin{matrix}
+1&+3&-1&+4\qquad&+1&+4&-1&-3
\\+2&+3&-2&+4\qquad&+2&+4&-2&-3
\end{matrix}$
&2$\left[ 2\right] $
&2$\left[ \left( 2\right) 2\right] $
\\\hline
2x2.48
&$\begin{matrix}
+1&+3&-1&+4\qquad&+1&+4&+2&+3
\\+1&-4&+2&-3\qquad&+2&+4&-2&+3
\end{matrix}$
&2$\left[ \right] $
&2$\left[ \left( 2\right) 2\right] $
\\\hline
2x2.49
&$\begin{matrix}
+1&+3&-1&+4\qquad&+1&+4&+2&+3
\\+1&-4&+2&-3\qquad&+2&+4&-2&-3
\end{matrix}$
&1$\left[ 2\right] $
&1$\left[ 2,4\right] $
\\\hline
2x2.50
&$\begin{matrix}
+1&+3&-1&+4\qquad&+1&+4&+2&-3
\\+1&-4&+2&+3\qquad&+2&+4&-2&+3
\end{matrix}$
&1$\left[ 4\right] $
&1$\left[ \left( 2\right) 2,4\right] $
\\\hline
2x2.51
&$\begin{matrix}
+1&+3&+2&+4\qquad&+1&-3&+2&-4
\\+1&+4&+2&+3\qquad&+1&-4&+2&-3
\end{matrix}$
&2$\left[ 2\right] $
&2$\left[ 2,4\right] $
\\\hline
2x2.52
&$\begin{matrix}
+1&+3&+2&+4\qquad&+1&-3&+2&-4
\\+1&+4&+2&-3\qquad&+1&-4&+2&+3
\end{matrix}$
&1$\left[ \left( 2\right) 2\right] $
&1$\left[ \left( 2\right) 2,4\right] $
\\\hline
\end{tabular}
\]

}  

\end{document}